\documentclass[12pt,
reqno]{amsart}

\usepackage[cp1251]{inputenc}
\usepackage{amsmath,amsxtra,amssymb,eufrak}
\usepackage[all]{xy}
\usepackage[active]{srcltx} 
\usepackage{mathrsfs}
\usepackage[english]{babel}

\newtheorem{thm}{Theorem}[section]
\newtheorem{lm}{Lemma}[section]
\newtheorem{pr}{Proposition}[section]

\theoremstyle{definition}

\newtheorem{df}{Definition}[section]
\newtheorem{rem}{Remark}[section]
\numberwithin{equation}{section}

\DeclareMathOperator{\Der}{Der}
\DeclareMathOperator{\Coker}{Coker}
\DeclareMathOperator{\Ker}{Ker}
\newcommand*{\ptens}[1]{\mathop{\widehat\otimes}_{#1}}

\let \al         =\alpha
\let \be         =\beta

\let \de         =\delta

\let \te         =\theta        
\let \io         =\iota

\let \la         =\lambda
\let \si         =\sigma

\let \De         =\Delta

\title{The Arens--Michael envelope of a solvable Lie algebra is a homological epimorphism}
\author{O. Yu. Aristov}
\address{Institute for Advanced Study in Mathematics of Harbin Institute of Technology, Harbin 150001, China;
\newline\indent
Suzhou Research Institute of Harbin Institute of Technology, Suzhou 215104, China}
\email{aristovoyu@inbox.ru}

\begin{document}
\begin{abstract}
The Arens--Michael envelope of the universal enveloping algebra of a finite-dimensional complex Lie algebra is a homological epimorphism if and only if the Lie algebra is solvable. The necessity was proved by Pirkovskii in [Proc. Amer. Math. Soc. 134, 2621--2631, 2006]. We prove the sufficiency.
\end{abstract}
 \maketitle
\markright{The Arens--Michael envelope of a solvable Lie algebra}

\section*{Introduction}

We are interested in the question under what conditions, for a given finite-dimensional complex Lie algebra~$\mathfrak{g}$, the universal completion (called the Arens--Michael envelope) of  $U(\mathfrak{g})$ (its universal enveloping algebra) is a homological epimorphism. This question is implicitly contained in a paper of J.\,Taylor published in early 70s \cite{T2} and explicitly formulated by Pirkovskii in \cite[\S\,9, Problem~1]{Pir_stbflat}.

\subsection*{Homological epimorphisms}
Flat homomorphisms play an important role in algebra and analysis but in some respects the requirement of flatness is too strong and must be replaced by a weaker condition. The concept of homological epimorphism was introduced by J.\,Taylor \cite{T2} as a weakened form of flatness under the name of ``absolute  localization'' and has been rediscovered several times  in different contexts and under different names --- ``lifting'', ``stably flat homomorphism'', ``isocohomological morphism'', ``homotopy epimorphism''; see references in \cite[Remark 3.16]{AP}. An important property for applications is that homological epimorphisms preserve Hochschild homologies and cohomologies  \cite[Propositions 1.4 and 1.7]{T2}. We are interested in the functional analytic version of this concept that goes back to Taylor and consider homological epimorphisms of complete locally convex algebras with jointly continuous multiplication (we call them $\mathop{\widehat\otimes}$-algebras); see Definition~\ref{Homoepi}.

An example of the use of homological epimorphisms is contained in Meyer's preprint \cite{Me04}, where  Connes's calculation of the cyclic cohomologies of smooth non-commutative tori in \cite[\S\,6]{Co85} is simplified. In addition, it was noted by Pirkovskii in~\cite{Pi99} that flat homomorphisms are not sufficient for characterization of open embeddings of Stein varieties but a weak version of the notion of homological epimorphism is adequate for this task; for an extension of results to general Stein spaces see \cite{AP,BBK18}. Another motivation for studying of  homological epimorphisms in the context of functional analysis comes from non-commutative spectral theory; see  the introduction to \cite{Pir_stbflat} for a detailed discussion and interpretation  of Taylor's ideas.

\subsection*{Arens--Michael envelopes}
The idea to study Arens--Michael envelopes and, in particular, to look for conditions under which they are homological epimorphisms also belongs to Taylor. Recall that the Arens--Michael envelope of a topological algebra~$A$ is the universal completion with respect to the class of Banach algebras (equivalently, the completion relative to the locally convex topology determined by all continuous submultiplicative seminorms on~$A$); see details in \S\,\ref{s:HASPAMe}.   Roughly speaking, this completion is responsible for continuous representations of~$A$ on Banach spaces. A prototype example of an Arens--Michael envelope is the embedding of the polynomial algebra $\mathbb{C}[z_1,\ldots,z_n]$ to the algebra of entire functions $\mathcal{O}(\mathbb{C}^n)$.

Taylor found explicit descriptions of the Arens--Michael envelopes of certain finitely-generated associative algebras and proved that in some cases the envelope is a homological epimorphism; also, he provided first counterexamples~\cite{T2}.

\subsection*{History and the main result}

Let $U(\mathfrak{g})$ be the universal enveloping algebra of a finite-dimensional complex Lie algebra  $\mathfrak{g}$.
The question of whether the Arens--Michael envelope homomorphism $U(\mathfrak{g})\to \widehat U(\mathfrak{g})$ is a homological epimorphism has been studied since 70s. The affirmarive answer in the case when $\mathfrak{g}$ is abelian, and the negative answer in the case when it is semisimple, was given by Taylor in~\cite{T2}.  Three decades later, in~\cite{Pi4}, Pirkovskii showed that the \emph{solvability a necessary condition}. At the same time, the study of the nilpotent case was started by Dosiev \cite{Do03,Do09}. It was continued by Pirkovskii in~\cite{Pir_stbflat} and recently completed by the author in~\cite{ArAMN}, where it was shown that the nilpotency is a sufficient condition. Thus, only the case when $\mathfrak{g}$ is solvable but not nilpotent remained uninvestigated (with one exception --- the two-dimensional non-abelian algebra~\cite{Pir_stbflat}). Here we prove  that the \emph{solvability is also a sufficient condition};  see Theorem~\ref{AMheiffsol}.  Thus, the final answer is that $U(\mathfrak{g})\to \widehat U(\mathfrak{g})$ is a homological epimorphism if and only if~$\mathfrak{g}$ is solvable.

The following papers of the author precede this work.

\begin{itemize}
\item
In \cite{ArAMN}, the nilpotent case is considered. The structure of $\widehat U(\mathfrak{g})$ is described and proved that  the Arens--Michael envelope is a homological epimorphism (using a method different from that applied in this article).
\item
In \cite{ArAnF}, the structure of the algebra of analytic  functionals on a connected complex Lie group is discussed. In particular, the results in~\cite{ArAnF} can be applied to $\widehat U(\mathfrak{g})$ for a general finite-dimensional Lie algebra. Some improvements is obtained in \cite{AHHFG}.
\item
A decomposition into an iterated analytic smash products (which is essential here) is constructed in \cite{Ar_smash}.
\item
The decomposition mentioned above exists not only for the Arens--Michael envelope but for other  completions. Some preliminary results on these completions is proved in \cite{ArPiLie,ArLfd}. Formally, \cite{Ar_smash} is based on them but, in fact, the case of the Arens--Michael envelope can be studied without reference to \cite{ArPiLie,ArLfd}.
\end{itemize}

Our main result is of independent interest but, on the other hand, this text is a part of a big project on completions of universal enveloping algebras and algebras of analytic functionals. Homological properties related with the first topic is a subject of this paper. The author plans to study homological properties of algebras of analytic functionals and their completions in a separate article.

The result was announced at the seminar ``Algebras in Analysis'' at Moscow State University in 2018. The author later discovered a gap in the proof proposed at that time. This text presents a different argument.

\tableofcontents

\section{Main ideas of the proof}

The proof is based on two circles of sets: the author's results on decomposing of $\widehat U(\mathfrak{g})$ into an iterated analytic smash product obtained in~\cite{Ar_smash} and  a powerful technique of relative homological epimorphisms
developed by Pirkovskii in~\cite{Pir_qfree}. In fact, the analytic part of preliminary work is contained in~\cite{Ar_smash} and preceding articles.   The technique of Pirkovskii, which is applied below, is more algebraic in nature.

When studying homological epimorphisms, the question of finding a convenient projective resolution is important. To construct such a resolution one can use a decomposition of $U(\mathfrak{g})$ into an iterated smash product.
In group homology, a  construction of a resolution of the trivial module of a semi-direct product  is  well known (see, e.g., \cite[Chapter\,5, \S\,2, pp.\,250--251]{Ku06}). It is not hard to generalize this construction to cocommutative  Hopf algebras and transfer it to the context of functional analysis. (This topic will be discussed elsewhere.) Unfortunately, it is not clear whether or not the tensor product functor in~\eqref{tenprfu} sends such projective resolution to a similar resolution. In \cite{Pir_qfree},  to overcome  this kind of difficulties, it was proposed to use relative homological homomorphisms instead of usual. Below we show  that, with appropriate modification, this idea works well in our case.

\subsection*{Outline of Pirkovskii's results}
In  \cite[Theorem 9.12]{Pir_qfree}, Pirkovskii proved that the Arens--Michael envelope is a homological epimorphism for a number of finitely-generated non-commutative algebras. In this paper, we are interested only in generalizing of Parts~(vi) and~(vii) of the above theorem, the parts that concern  the two-dimensional non-abelian and three-dimensional Heisenberg Lie algebra. This section contains an extract of the argument in these cases. We postpone the necessary definitions until \S\,\ref{s:rsfhe}.

In both the two-dimensional and three-dimensional cases considered in~\cite{Pir_qfree}, a decomposition of $\mathfrak{g}$ into a semidirect sum induces a decomposition of $U(\mathfrak{g})$ into an Ore extension of a commutative algebra~$R$. Namely, $U(\mathfrak{g})$ is isomorphic to $\mathbb{C}[y][x; \de]$ with $\de=y\frac d{dy}$ in the first case and $\mathbb{C}[y,z][x; \de]$ with $\de=z\frac \partial{\partial y} $  in the second. As a corollary, $\widehat U(\mathfrak{g})$ can be represented as an analytic Ore extension $\mathcal{O}(\mathbb{C},R_\de; \widehat\de)$, where $R_\de$ is a certain completion of $R$ and $\widehat\de$ is the extension of $\de$ to $R_\de$ \cite[Theorem 5.1]{Pir_qfree}. (Here and below $\mathcal{O}$ stands for holomorphic functions.)

Pirkovskii used a two-step argument. First he showed that $R\to R_\de$ is a homological epimorphism and next that the Arens--Michael envelope homomorphism $\iota\!:U(\mathfrak{g})\to\widehat U(\mathfrak{g})$ is a one-sided relative homological epimorphism with respect to~$R_\de$. Both conditions are included into the hypotheses of the following result.

\begin{thm}\label{PiTh}
\cite[Theorem 9.1(ii)]{Pir_qfree}
Let $(f,g)\!: (A,R)\to (B,S)$ be an $R$-$S$-homomorphism from an $R$-$\mathop{\widehat\otimes}$-algebra $A$ to an $S$-$\mathop{\widehat\otimes}$-algebra $B$. Suppose that

\emph{(1)}~$g$ is a homological epimorphism;

\emph{(2)}~$f$ is a left or right relative homological epimorphism;

\emph{(3)}~$A$ is projective in $R{\mbox{-}\!\mathop{\mathsf{mod}}}$ and $B$ is projective in $S{\mbox{-}\!\mathop{\mathsf{mod}}}$;

\emph{(4)}~$A$ is an $R$-$\mathop{\widehat\otimes}$-algebra of $(f_2)$-finite type.

Then $f$ is a homological epimorphism.
\end{thm}

Returning to low-dimensional Lie algebras, consider $R$-$R_\de$-homomorphism
$$
(\io,g)\!: (U(\mathfrak{g}),R)\to (\widehat U(\mathfrak{g}),R_\de),
$$
where  $g$ denotes the homomorphism $R\to R_\de$. To prove that $\io$ is a homological epimorphism it suffices to check Conditions~(1)--(4) in Theorem~\ref{PiTh} for $(\io,g)$.

For Ore extensions and analytical Ore extensions, Condition~(3)  easily follows from definitions.
Note that Ore extensions have two useful properties: they are not only of $(f_2)$-finite type but also relatively quasi-free; see Proposition~\ref{Oeqfft} below. In particular, Condition~(4) holds. Checking of Conditions~(1) and~(2) is more challenging. To prove  Condition~(2) Pirkovskii used relatively quasi-free algebras. (Note that the use of the relative quasi-freeness was a significant innovation in this topic.) In particular, the following result is important.

\begin{thm}\label{AMerelhoep}
\cite[Theorem 7.6(ii)]{Pir_qfree}
Let $A$, $R$, $S$ be $\mathop{\widehat\otimes}$-algebras and $g\!: R \to S$ an epimorphism of $\mathop{\widehat\otimes}$-algebras and let $\iota_A\!: A \to \widehat A$ denote the Arens--Michael envelope. Suppose that $A$ is an $R$-$\mathop{\widehat\otimes}$-algebra and $\widehat A$ is an $S$-$\mathop{\widehat\otimes}$-algebra such that the pair $(\iota_A,g)$ is an $R$-$S$-homomorphism. Assume also that $A$ is relatively quasi-free over $R$. Then $\iota_A$ is a two-sided relative homological epimorphism.
\end{thm}

In the proof, it is essential  that $\iota_A\!: A \to \widehat A$ satisfies to the unique extension property for derivations (Property (UDE)); see the corresponding definition in \S\,\ref{s:UDE}.

Verifying Condition~(1) is arduous since $R\to R_\de$ is not an Arens--Michael envelope in general. To surmount this challenge Pirkovskii applied another technique, which we do not use in this paper; see details in the proof of  Parts~(vi) and~(vii) of Theorem~9.12 in~\cite{Pir_qfree}.

In the high-dimensional case, it is natural to use induction.
Specifically,  when an algebra $A$ has a chain of subalgebras $\mathbb{C}= R_0\subset  R_1 \subset\cdots\subset R_n = A$ such that an iterative application of Theorem~\ref{PiTh} is possible, we have a tool for proving that $A\to \widehat A$ is a homological epimorphism; see the introduction to~\cite{Pir_qfree}. This plan works well in some situations but  unfortunately cannot be directly applied to all high-dimensional solvable Lie algebras. Indeed, such an algebra admits an iterated semidirect sum decomposition,
$$
\mathfrak{g}=((\cdots (\mathfrak{f}_1 \rtimes \mathfrak{f}_2)\rtimes\cdots)\rtimes \mathfrak{f}_n,
$$
where $\mathfrak{f}_1,\ldots,\mathfrak{f}_n$ are $1$-dimensional. Moreover, this decomposition actually induces an decomposition of the universal enveloping algebra into  an iterated Ore extension; see, e.g., \cite[pp. 33--34, 1.7.11(iv)]{MR87}. But, in this case, $R_{n-1}\to (R_{n-1})_\de$ is not usually an Arens--Michael envelope and, moreover, $(R_{n-1})_\de$ cannot be represented as an analytic Ore extension. So we need modifications.

\subsection*{Plan of the proof}

Our strategy is to apply Theorem~\ref{PiTh} using iterations. But the scheme proposed in~\cite{Pir_qfree} needs to be changed, as is clear from the previous discussion.

1.~We use iterated analytic smash products instead of analytic Ore extensions; see Theorem~\ref{AMede}. (An analytic Ore extension with trivial twisting is an analytic smash product but when the twisting is non-trivial this is not the case; thus most of parts of Theorem 9.12 in~\cite{Pir_qfree} are not covered by our approach.) Some preliminaries on Hopf $\mathop{\widehat\otimes}$-algebras, analytic smash products and Arens--Michael envelopes are contained in~\S\,\ref{s:HASPAMe}.

2.~We use completions of $\mathbb{C}[z]$ other than $\mathcal{O}(\mathbb{C})$, namely, the algebras $\mathfrak{A}_s$ defined in~\eqref{faAsdef}. The essential point is that $\mathcal{O}(\mathbb{C})\to\mathfrak{A}_s$ and some induced homomorphisms of analytic smash product satisfy Property (UDE). To prove theses facts we apply an auxiliary result on pushouts; see Theorem~\ref{UDEinh} and preparatory results in~\S\,\ref{s:HASPAMe}.

3.~We get an Arens--Michael envelope only on the final step of iteration. So we have to replace Theorem~\ref{AMerelhoep} by a more general assertion on homomorphisms satisfying Property (UDE); see Theorem~\ref{Brelhoep}. Combining it with a decomposition result essentially proved in~\cite{Ar_smash} (Theorem~\ref{AMede}), we finally deduce Theorem~\ref{AMheiffsol}.

\section{Hopf $\mathop{\widehat\otimes}$-algebras, analytic smash products and Arens--Michael envelopes}
\label{s:HASPAMe}

Consider the bifunctor $(-)\mathbin{\widehat{\otimes}} (-)$ of complete projective tensor product on the category of complete locally convex spaces and Hopf algebras in the corresponding symmetric monoidal category. We call them \emph{Hopf $\mathbin{\widehat{\otimes}}$-algebras} (read `topological Hopf algebras'); see \cite{Lit} or \cite{Pir_stbflat}.
Also, we deal with $\mathbin{\widehat{\otimes}}$-algebras and $\mathbin{\widehat{\otimes}}$-(bi)modules, i.e.,
complete locally convex associative algebras and (bi)modules with jointly continuous multiplication. We assume that each algebra contains an identity and each module is unital.  A $\mathbin{\widehat{\otimes}}$-(bi)module over a $\mathbin{\widehat{\otimes}}$-algebra~$A$ is referred as an $A$-$\mathbin{\widehat{\otimes}}$-(bi)module. We also assume that $\mathop{\widehat\otimes}$-algebra homomorphisms  preserve identity   and are continuous (as well as  $\mathbin{\widehat{\otimes}}$-module morphisms).

\subsection*{Generalized Sweedler notation}
In the Hopf algebra theory, a  notation introduced by Sweedler is widely used instead of the tensor notation in cases when the latter is not convenient. It was noted in  \cite[\S\,2.4]{Ak08} that this notation can be  generalized to topological Hopf algebras. A version of the generalized Sweedler notation sufficient for our purposes is described in~\cite{Ar_smash} and we briefly recall it here.

In the classical Sweedler notation, the comultiplication $\De$ on a Hopf algebra is written~as
\begin{equation}\label{SwDe}
\De(h)=\sum h_{(1)}\otimes h_{(2)},
\end{equation}
where an arbitrary representation of $\De(h)$ is taken.

Before using this type of notation  for Hopf $\mathop{\widehat\otimes}$-algebras note that, in the Fr\'echet space context, we can treat~\eqref{SwDe} not as a finite sum but as a convergent series. However, in general  this is not always possible and we write $\De(h)$ as the limit of a net of finite sums of elementary tensors:
$$
\De(h)=\lim_\nu \sum_{i=1}^{n_\nu} h_{(1)}^{\nu,i}\otimes h_{(2)}^{\nu,i}.
$$
In the case of a Hopf $\mathop{\widehat\otimes}$-algebra, \eqref{SwDe} is a short form of the  above formula.

For example, in the full form, the coassociativity axiom, $(1\otimes \De)\De(h)=(\De\otimes 1)\De(h)$, can be written as
\begin{multline}\label{detcoass}
\lim_\nu \left(\sum_{i=1}^{n_\nu} h_{(1)}^{\nu,i}\otimes \left(\lim_\mu \sum_{j=1}^{m_\mu} (h_{(2)}^{\nu,i})_{(1)}^{\mu,j}\otimes (h_{(2)}^{\nu,i})_{(2)}^{\mu,j}\right)\right)=\\
=
\lim_\nu \left(\left(\lim_\la \sum_{j=1}^{l_\la} (h_{(1)}^{\nu,i})_{(1)}^{\la,j}\otimes (h_{(1)}^{\nu,i})_{(2)}^{\la,j}\right)\otimes\sum_{i=1}^{n_\nu} h_{(2)}^{\nu,i}\right),
\end{multline}
where
$$
\De(h_{(1)}^{\nu,i})=\lim_\la \sum_{j=1}^{l_\la} (h_{(1)}^{\nu,i})_{(1)}^{\la,j}\otimes (h_{(1)}^{\nu,i})_{(2)}^{\la,j}\quad
\text{and}\quad
\De(h_{(2)}^{\nu,i})=\lim_\mu \sum_{j=1}^{m_\mu} (h_{(2)}^{\nu,i})_{(1)}^{\mu,j}\otimes (h_{(2)}^{\nu,i})_{(2)}^{\mu,j}.
$$
In the generalized Sweedler notation, it takes the same form as in the classical:
$$
\sum h_{(1)}\otimes \left(\sum(h_{(2)})_{(1)}\otimes (h_{(2)})_{(2)}\right)= \left(\sum(h_{(1)})_{(1)}\otimes (h_{(1)})_{(2)}\right)\otimes h_{(2)}.
$$
By the iterated limit theorem \cite[Chapter~2, p.\,69, Theorem~4]{Kel}, both iterated limits in~\eqref{detcoass} can be replaced by a simple limit of a net. In particular, this means that the iteration  of the comultipication can be written as
$$
(1\otimes \De)\De(h)=\sum h_{(1)}\otimes h_{(2)}\otimes h_{(3)},
$$
where a simple limit is also implied.

Also, the antipode axiom for a Hopf $\mathop{\widehat\otimes}$-algebra takes the standard form
$$
\sum S(h_{(1)})h_{(2)}=\varepsilon(h)1=\sum h_{(1)}S(h_{(2)}).
$$
This formula stands for
$$
\lim_\nu \sum_{i=1}^{n_\nu} S(h_{(1)}^{\nu,i})h_{(2)}^{\nu,i}=\varepsilon(h)1=\lim_\nu \sum_{i=1}^{n_\nu} h_{(1)}^{\nu,i}S(h_{(2)}^{\nu,i}).
$$

\subsection*{Analytic smash products}

In this section, we recall the necessary information on analytic smash products, which were introduced  in~\cite{Pi4}.
In the definitions we follow \cite{Ar_smash}, where some historical remarks can also be found.

Let $H$ be a Hopf $\mathbin{\widehat{\otimes}}$-algebra and $A$ a $\mathbin{\widehat{\otimes}}$-algebra that is a left $H$-$\mathbin{\widehat{\otimes}}$-module.
Then $A\mathbin{\widehat{\otimes}} A$ and $\mathbb{C}$ are left $H$-$\mathbin{\widehat{\otimes}}$-modules with respect to the actions given by
$$
h\cdot(a\otimes b)\!:=\sum (h_{(1)}\cdot a)\otimes  (h_{(2)}\cdot b)
\qquad (h\in H,\,a,b\in A);
$$
$$
h\cdot\lambda\!:=\varepsilon(h)\lambda, \qquad (h\in H,\,\lambda\in\mathbb{C}).
$$

Recall that~$A$ is called a (left)  \emph{$H$-$\mathbin{\widehat{\otimes}}$-module algebra} if the linearized multiplication $\mu\!:A\mathbin{\widehat{\otimes}} A\to A$ and the unit map $\mathbb{C}\to A$ are left $H$-module morphisms, i.e.,
\begin{equation}\label{Hmodalgcond}
h\cdot(ab)=\sum (h_{(1)}\cdot a)(h_{(2)}\cdot b)\quad\text{and}\quad h\cdot 1=\varepsilon(h)1\qquad(h\in H,\,a,b\in A).
\end{equation}

For a  $\mathbin{\widehat{\otimes}}$-algebra homomorphism $\psi\!:H\to B$ consider the
\emph{adjoint action} of $H$ on $B$:
\begin{equation}\label{adac}
h\cdot b \!:=\sum \psi(h_{(1)})\,b\,\psi(S(h_{(2)}))\qquad (h\in H,\,b\in B).
\end{equation}
It is easy to see that $B$ is an $H$-$\mathbin{\widehat{\otimes}}$-module algebra with respect to this action.

\begin{df}\label{PsmaDef}
Let $H$ be a Hopf $\mathbin{\widehat{\otimes}}$-algebra and $A$ an $H$-$\mathbin{\widehat{\otimes}}$-module algebra. The \emph{analytic smash product}  $A\mathop{\widehat{\#}} H$ is defined as a $\mathbin{\widehat{\otimes}}$-algebra endowed with $\mathbin{\widehat{\otimes}}$-algebra homomorphisms
$i\!:A\to A\mathop{\widehat{\#}} H$ and $j\!:  H\to A\mathop{\widehat{\#}} H$ such that the following conditions hold.

(A)~$i$ is an $H$-$\mathbin{\widehat{\otimes}}$-module algebra homomorphism with respect to the adjoint action associated with~$j$.

(B)~For every $\mathbin{\widehat{\otimes}}$-algebra $B$, $\mathbin{\widehat{\otimes}}$-algebra homomorphism $\psi\!:  H\to  B$  and $H$-$\mathbin{\widehat{\otimes}}$-module algebra homomorphism $\varphi\!: A\to B$ (with respect to the adjoint action associated with~$\psi$) there is a unique
$\mathbin{\widehat{\otimes}}$-algebra homomorphism  $\theta$ such that the diagram
\begin{equation*}
   \xymatrix{
 & A\mathop{\widehat{\#}} H \ar@{-->}[dd]^{\theta}& \\
A \ar[ur]^i \ar[dr]_\varphi&&H\ar[ul]_j\ar[dl]^\psi\\
& B &}
\end{equation*}
is commutative.
\end{df}

The explicit construction is as follows. The formula
\begin{equation}\label{exmusmpr}
(a\otimes h)(a'\otimes h')=\sum a(h_{(1)}\cdot a')\otimes h_{(2)}h'\qquad(h,h'\in H;\,a,a'\in A)
\end{equation}
determines a multiplication on $A\mathbin{\widehat{\otimes}}H$.
Endowed with the maps $i\!:a\mapsto a\otimes 1$ and $j\!:h\mapsto 1\otimes h$ and this multiplication, $A\mathbin{\widehat{\otimes}}H$  is an analytic smash product.

\begin{rem}\label{OeAsp}
In the case when $H=\mathcal{O}(\mathbb{C})$ and the action of $H$ is induced by a derivation,
the analytic smash product coincides with the  corresponding analytic Ore extension
\cite[Remarks~4.4 and~4.2]{Pir_qfree}. The same holds in the pure algebraic case, i.e., when the action of $\mathbb{C}[z]$ is induced by a derivation.
\end{rem}

The following lemma follows directly from the definitions.
\begin{lm}\label{homsmpr}
\cite[Lemma 3.10]{Ar_smash}
Let $H$ and $K$ be Hopf $\mathbin{\widehat{\otimes}}$-algebras, $R$ a $H$-$\mathbin{\widehat{\otimes}}$-module algebra and $S$ a $K$-$\mathbin{\widehat{\otimes}}$-module algebra. If  $\be\!:H\to K$ is a Hopf $\mathbin{\widehat{\otimes}}$-algebra homomorphism  and  $\al\!: R\to S$ is a $\mathbin{\widehat{\otimes}}$-algebra homomorphism  that is an  $H$-$\mathbin{\widehat{\otimes}}$-module morphism \emph{(}i.e.,
$\al(h\cdot r)=\be(h)\cdot\al(r)$ when $h\in H$, $r\in R$\emph{)}, then the formula
\begin{equation}\label{morsmpr}
\al\mathop{\widehat{\#}} \be\!:R\mathop{\widehat{\#}} H\to S\mathop{\widehat{\#}} K\!:r\otimes
h\mapsto \al(r)\otimes \be(h)
\end{equation}
determines a $\mathbin{\widehat{\otimes}}$-algebra homomorphism.
\end{lm}

\subsection*{Smash product decomposition of Arens--Michael envelopes}

In what follows we need a family of Hopf $\mathop{\widehat\otimes}$-algebras importance of which for finding Arens--Michael envelopes was discovered in  \cite{ArRC} and \cite{AHHFG}.

For $s\in[0,+\infty)$ put
\begin{equation}
 \label{faAsdef}
\mathfrak{A}_s\!:=\Bigl\{a=\sum_{n=0}^\infty  a_n z^n\! :
\|a\|_{r,s}\!:=\sum_{n=0}^\infty |a_n|\,\frac{r^n}{n!^s}<\infty
\quad\forall r>0\Bigr\},
\end{equation}
where $z$ is treated as a formal variable, and endow $\mathfrak{A}_s$ with the topology determined by the seminorm family $(\|\cdot\|_{r,s};\,r>0)$. Denote also $\mathbb{C}[[z]]$, the algebra of all formal power series in $z$, by $\mathfrak{A}_\infty$.

\begin{lm}\label{CzAsHh}
Let $s\in[0,+\infty]$. Then $\mathfrak{A}_s$ is a Hopf $\mathop{\widehat\otimes}$-algebra with respect to the operations continuously extended from $\mathbb{C}[z]$ and so the natural embedding $\be\!:\mathbb{C}[z]\to \mathfrak{A}_s$ is a Hopf $\mathop{\widehat\otimes}$-algebra homomorphism.
\end{lm}
Here  we assume that $\mathbb{C}[z]$ is endowed with the strongest locally convex topology and hence it is a $\mathop{\widehat\otimes}$-algebra with respect to this topology \cite[Proposition 2.3]{Pir_qfree}.
In what follows we assume that all associative algebras of countable dimension are endowed with the strongest locally convex topology.

The assertion of the lemma is proved in  \cite[Example 2.4]{AHHFG} when $s\ne\infty$ and can be directly checked when $s=\infty$.

Recall that an \emph{Arens--Michael algebra} is a complete topological algebra whose topology can be determined by a family of submultiplicative seminorms.
The \emph{Arens--Michael envelope} of a topological algebra $A$ is a pair $(\widehat A, \io_A)$, where $\widehat A$ is an Arens--Michael algebra and $\io_A$ is a continuous homomorphism $A \to \widehat A$ such that for every Arens--Michael algebra~$B$ and every continuous homomorphism $\varphi\!: A \to B$ there is a unique continuous homomorphism
$\widehat\varphi\!:\widehat A\to B$ making the diagram
\begin{equation*}
  \xymatrix{
A \ar[r]^{\io_A}\ar[rd]_{\varphi}&\widehat A\ar@{-->}[d]^{\widehat\varphi}\\
 &B\\
 }
\end{equation*}
commutative. More concretely, we can take as $\io_A$ the completion homomorphism with respect to the topology determined by all possible continuous submultiplicative seminorms on~$A$.

In the Lie algebra case,  we can replace  $U(\mathfrak{g})$ by~$\mathfrak{g}$ in the diagram (assuming that $\varphi$ is a Lie algebra homomorphism) and say that  $\widehat U(\mathfrak{g})$ is the Arens--Michael envelope of~$\mathfrak{g}$.

Our proof of the main result, Theorem~\ref{AMheiffsol}, is based on the following structural theorem, which is easily implied by  results
in~\cite{Ar_smash}.

\begin{thm}\label{AMede}
Let $\mathfrak{g}$ be a finite-dimensional solvable complex Lie algebra. Then there is an iterated semidirect sum decomposition \begin{equation}\label{fgdec}
\mathfrak{g}=((\cdots (\mathfrak{f}_1 \rtimes \mathfrak{f}_2)\rtimes\cdots)\rtimes \mathfrak{f}_n,
\end{equation}
where $\mathfrak{f}_1,\ldots,\mathfrak{f}_n$ are $1$-dimensional, and a non-increasing sequence $i_1,\ldots,i_{n}$  in $[0,\infty]$
such that
 \begin{equation}\label{expfsmp}
\widehat{U}(\mathfrak{g})\cong(\cdots (\mathfrak{A}_{i_1} \mathop{\widehat{\#}}\mathfrak{A}_{i_2})
\mathop{\widehat{\#}}\cdots)\mathop{\widehat{\#}}\mathfrak{A}_{i_n}
\end{equation}
Moreover,  the homomorphism $U(\mathfrak{g}) \to \widehat{U}(\mathfrak{g})$ is compatible with the decomposition
$$
U(\mathfrak{g})\cong(\cdots (U(\mathfrak{f}_1) \mathop{\#}U(\mathfrak{f}_2))\mathop{\#} \cdots)\mathop{\#} U(\mathfrak{f}_n)
$$
induced by~\eqref{fgdec}.
\end{thm}
Here the compatibility of decompositions means that at each step we have a smash product with a homomorphism of the form $\be$ considered in Lemma~\ref{CzAsHh} as described in Lemma~\ref{homsmpr}. (Note that $U(\mathfrak{f}_i)\cong \mathbb{C}[z]$ for every $i$.)
\begin{proof}
We use reduction to algebras of analytic functionals.
Let $G$ be a simply connected complex Lie group whose Lie algebra is isomorphic to $\mathfrak{g}$. Also, let ${\mathscr A}(G)$ denote the corresponding algebra of analytic functionals (the strong dual space of the space holomorphic functions on $G$ endowed with the  convolution multiplication) and $\widehat{\mathscr A}(G)$ the Arens--Michael envelope of ${\mathscr A}(G)$. Then the natural embedding $U(\mathfrak{g})\to{\mathscr A}(G)$ induces a topological isomorphism
$\widehat U(\mathfrak{g})\to\widehat{\mathscr A}(G)$ \cite[Proposition 2.1]{ArAMN}. So we can apply results on $\widehat{\mathscr A}(G)$, Theorems~4.4,~6.5 and~6.4  in~\cite{Ar_smash}, which give the existence of the decomposition in~\eqref{fgdec} and~\eqref{expfsmp}, and the compatibility with the decomposition of $U(\mathfrak{g})$, respectively.
\end{proof}

\subsection*{Smash product and pushouts}

We use the results in this section in the proof of Theorem~\ref{UDEinh}.

\begin{pr}\label{UDElpu}
Let $H$ be a Hopf $\mathbin{\widehat{\otimes}}$-algebra and $\al\!: R\to S$ an $H$-$\mathbin{\widehat{\otimes}}$-module algebra homomorphism. If $\al$ has dense range, then
\begin{equation*}
   \xymatrix{
  R  \ar[d]_\al\ar[r]_{i_R}&R \mathop{\widehat{\#}} H \ar[d]^{\al\mathop{\widehat{\#}} 1} \\
S \ar[r]_{i_S}  &S \mathop{\widehat{\#}} H,
}
\end{equation*}
where $i_R$ and $i_S$ are the canonical homomorphisms,
is a pushout diagram in the category of $\mathbin{\widehat{\otimes}}$-algebras.
\end{pr}
\begin{proof}
Let $\varphi\!: S\to B$ and $\chi\!:R \mathop{\widehat{\#}} H \to B$ be $\mathop{\widehat\otimes}$-algebra homomorphisms such that $\varphi\al=\chi i_R$. We want to use the universal property in Definition~\ref{PsmaDef}. Put $\psi\!:=\chi j$, where $j\!:H\to R \mathop{\widehat{\#}} H$ is the canonical homomorphism. It suffices to find a $\mathop{\widehat\otimes}$-algebra homomorphism $\te$ making the diagram
\begin{equation*}
\xymatrix@C=20pt{
   R \ar[rr]^{i_R}\ar[d]_{\al} && R \mathop{\widehat{\#}} H \ar[d]_{\al\mathop{\widehat{\#}} 1}\ar[ddr]^{\chi}& H\ar[l]^j\ar[dd]^\psi \\
 S\ar[rr]^{i_S}\ar[drrr]_{\varphi}&& S \mathop{\widehat{\#}} H\ar@{-->}[dr]_{\te}\\
     &&&B
 }
 \end{equation*}
commutative and prove that it is unique.

If $h\in H$ and  $r\in R$, then by the definition of the adjoint action in~\eqref{adac}, we have that
$$
h\cdot \varphi(\al(r))=\sum \psi(h_{(1)})\,\varphi(\al(r))\,\psi(S(h_{(2)}))=\chi \left(\sum j(h_{(1)})\,i_R(r)\,j(S(h_{(2)}))\right)=\chi(i_R(h\cdot r)).
$$
(The last equality follows from the fact that $i_R$ is an $H$-$\mathbin{\widehat{\otimes}}$-module morphism with respect to the adjoint action.) Thus $h\cdot \varphi(\al(r))=\varphi(\al(h\cdot r))$. Hence $\varphi\al$ is an $H$-$\mathbin{\widehat{\otimes}}$-module morphism.

Further, by the hypothesis, $\al$ is also an $H$-$\mathbin{\widehat{\otimes}}$-module morphism and so is $\varphi$ since the range of $\al$ is dense.
Hence, by the universal property of smash products written for $ S \mathop{\widehat{\#}} H$, there is a unique  $\mathop{\widehat\otimes}$-algebra homomorphism $\te$ such that $\te i_S=\varphi$ and $\te(\al\mathop{\widehat{\#}} 1)j=\psi$.

Since $\varphi\al$ is an $H$-$\mathbin{\widehat{\otimes}}$-module algebra homomorphism, it follows from the universal property of smash products written for $R\mathop{\widehat{\#}} H$ that a homomorphism $\chi$ such that $\chi i_R=\varphi\al$ and $\chi j=\psi$ is unique. Therefore $\te(\al\mathop{\widehat{\#}} 1)=\chi$.

To complete the proof we need to show that $\te$ such that $\te i_S=\varphi$ and $\te(\al\mathop{\widehat{\#}} 1)=\chi$ is unique. If $\te'$ is another $\mathop{\widehat\otimes}$-algebra homomorphism with these properties, then $\te'(\al\mathop{\widehat{\#}} 1)i_R=\varphi\al$ and $\te'(\al\mathop{\widehat{\#}} 1)j=\psi$. The uniqueness statement above implies that $\te'=\te$.
\end{proof}

\begin{pr}\label{UDErpu}
Let $\be\!:H\to K$ be a Hopf $\mathbin{\widehat{\otimes}}$-algebra homomorphism  and $S$ a $K$-$\mathbin{\widehat{\otimes}}$-module algebra.  If $\be$ has dense range, then
\begin{equation*}
   \xymatrix{
S \mathop{\widehat{\#}} H  \ar[d]_{1\mathop{\widehat{\#}}\be}& H \ar[d]^\be\ar[l]^{j_H} \\
S \mathop{\widehat{\#}} K  & K\ar[l]^{j_K},
}
\end{equation*}
where $j_H$ and $j_K$ are the canonical homomorphisms,
is a pushout diagram in the category of $\mathbin{\widehat{\otimes}}$-algebras.
\end{pr}
\begin{proof}
Let $\psi\!: K\to B$ and $\chi\!:S \mathop{\widehat{\#}} H \to B$ be $\mathop{\widehat\otimes}$-algebra homomorphisms  such that $\psi\be=\chi j_H$. We use the universal property in Definition~\ref{PsmaDef} as in the proof of Proposition~\ref{UDElpu}. Put $\varphi\!:=\chi i$, where $i\!:S\to S \mathop{\widehat{\#}} H$ is the canonical homomorphism. It suffices to find a $\mathop{\widehat\otimes}$-algebra homomorphism $\te$ making the diagram
\begin{equation*}
   \xymatrix{
S\ar[dd]_\varphi\ar[r]^{i} &S \mathop{\widehat{\#}} H  \ar[ddl]_\chi\ar[d]^{1\mathop{\widehat{\#}}\be}&& H \ar[d]^\be\ar[ll]_{j_H} \\
&S \mathop{\widehat{\#}} K\ar@{-->}[dl]^{\te}  && K\ar[ll]^{j_K}\ar[dlll]^\psi,\\
B&&&
}
\end{equation*}
commutative  and prove that it is unique.

Recall that $S$ is endowed with the action $h\cdot s=\be(h)\cdot s$. If $h\in H$ and  $s\in S$, then by the definition of the adjoint action, we have
\begin{multline*}
\be(h)\cdot \varphi(s)=\sum \psi(\be(h_{(1)}))\,\varphi(s)\,\psi(\be(S(h_{(2)})))=\\
\chi \left(\sum j_H(h_{(1)})\,i(s)\,j_H(S(h_{(2)}))\right)=\chi( i(\be(h)\cdot s)).
\end{multline*}
(The last equality follows from the fact that $i$ is an $H$-$\mathbin{\widehat{\otimes}}$-module morphism with respect to the adjoint action associated with $j_H$.) Hence $k\cdot \varphi(s)=\varphi(k\cdot s)$ when $k=\be(h)$. Therefore $\varphi$ is a $K$-$\mathbin{\widehat{\otimes}}$-module morphism since the range of $\be$ is dense.

Hence, by the universal property of smash products written for $S\mathop{\widehat{\#}} K$, there is a unique $\mathop{\widehat\otimes}$-algebra homomorphism $\te$ such that $\te j_K=\psi$ and $\te(1\mathop{\widehat{\#}}\be)i=\varphi$. The first equality implies that $\te(1\mathop{\widehat{\#}}\be)j_H=\psi\be$.

Being a $K$-$\mathbin{\widehat{\otimes}}$-module morphism, $\varphi$ is also an $H$-$\mathbin{\widehat{\otimes}}$-module morphism. Then the universal property of smash products written for $S\mathop{\widehat{\#}} H$ implies that a homomorphism $\chi$ such that $\chi i=\varphi$ and $\chi j_H=\psi\be$ is unique. Therefore $\te(1\mathop{\widehat{\#}}\be)=\chi$.

To complete the proof we need to show that $\te$ such that $\te j_K=\psi$ and $\te(1\mathop{\widehat{\#}}\be)=\chi$ is unique. If $\te'$ is another $\mathop{\widehat\otimes}$-algebra homomorphism with these properties, then  $\te'(1\mathop{\widehat{\#}}\be)i=\varphi$ and $\te'(1\mathop{\widehat{\#}}\be)j_H=\psi\be$. The uniqueness statement above implies that $\te'=\te$.
\end{proof}

\section{Unique extension property for derivations}
\label{s:UDE}

We denote the vector space of continuous derivations from a $\mathop{\widehat\otimes}$-algebra $A$ to an $A$-$\mathop{\widehat\otimes}$-bimodule $X$ by $\Der(A,X)$. It is clear that a $\mathop{\widehat\otimes}$-algebra homomorphism $\varphi\!:A\to B$ induces the linear map
$$
\widetilde\varphi_X\!:\Der(B,X)\to \Der(A,X)\!:D\mapsto D\varphi.
$$

\begin{df}\label{UDEdef}
We say that a homomorphism $\varphi\!:A\to B$ of $\mathop{\widehat\otimes}$-algebras satisfies   \emph{Property (UDE)} (the unique  extension property for derivations) if $\widetilde\varphi_X\colon\Der (B, X)\to \Der (A, X)$ is bijective for each $B$-$\mathop{\widehat\otimes}$-bimodule $X$.
\end{df}

It is proved in \cite{Pir_qfree} that every Arens--Michael envelope satisfies Property~(UDE).

In what follows we denote the pushout of $\mathop{\widehat\otimes}$-algebra homomorphisms $A\to B_1$ and $A\to B_2$ using the relative free algebra notation, i.e., as $B_1\mathbin{\ast_A} B_2$.

\begin{pr}\label{pushoex}
The category of $\mathop{\widehat\otimes}$-algebras is cocomplete.  
\end{pr}
\begin{proof}
It suffices to show that coproducts and coequalizers of pairs exist.

To prove the existence of coproducts we first note that for every locally convex space $E$ there is a unital tensor $\mathop{\widehat\otimes}$-algebra $T(E)$ This means that the following universal property is satisfied: there is a continuous linear map $\mu\!:E\to T(E)$ such that for every unital $\mathop{\widehat\otimes}$-algebra~$B$ and every continuous linear map $\psi\!: E \to B$ there exists a unique unital continuous homomorphism $\widehat\psi\!:T(E)\to B$ making the diagram
\begin{equation*}
  \xymatrix{
E \ar[r]^{\mu}\ar[rd]_{\psi}&T(E)\ar@{-->}[d]^{\widehat\psi}\\
 &B\\
 }
\end{equation*}
commutative. In the category of non-unital $\mathop{\widehat\otimes}$-algebras, a tensor algebra $T_0(E)$ is proved to exist  in \cite[Theorem 1]{Va01} and we can take the unitization of $T_0(E)$ as $T(E)$.

Now let $(A_i)$ be a family of unital $\mathop{\widehat\otimes}$-algebras.
Then the coproduct  can be represented as the completion of the quotient of $T(\oplus_i A_i)$ by the closed two-sided ideal generated by all elements of the form $a_i\otimes a'_i - a_ia'_i$ and  $1 - 1_{A_i}$, where  $a_i,a'_i\in A_i$ , $1_{A_i}$ is the identity in $A_i$ and $i$ runs through all possible values; cf. \cite[Proposition 4.2]{Pi15}.

To prove the existence of coequalizers suppose that $\varphi,\psi\!:A\to B$ are parallel $\mathop{\widehat\otimes}$-algebra homomorphisms. Denote by $C$ the completion of the quotient $B/I$, where $I$ is the closed two-sided ideal generated by $\{\varphi(a)-\psi(a);\,a\in A\}$. (We need the completion since the quotient space may be non-complete.) 

We claim that $(C,\tau)$, where $\tau\!:B\to C$ is the naturally defined homomorphism, is a coequalizer.
Indeed, let $\chi\! : B\to D$ be a $\mathop{\widehat\otimes}$-algebra homomorphism such that $\chi\varphi= \chi\psi $. Then $\chi(\varphi(a)-\psi(a))=0$ for every $a$ and so $\chi$ maps $I$ to $0$. Hence $\chi$ induces a continuous homomorphism $\chi'\!:B/I\to D$ such that $\chi=\chi'\tau$. Since $\chi'$ extends continuously to $C$, the claim is proved. 
\end{proof}

\begin{rem}
A preliminary version of this article referenced Proposition~B.1 in the author's paper~\cite{AHHFG}, whose proof is generally incorrect.
The author is grateful to the referee for pointing out a mistake in \cite{AHHFG} and observing that a correct proof can be found in \cite{Va01}.
\end{rem}

In particular, we have that pushouts  of $\mathop{\widehat\otimes}$-algebras exist. The following result is an analytic version of Proposition 5.2 in~\cite{BerDic}.

\begin{pr}\label{DDPsta}
In the  category of unital  $\mathop{\widehat\otimes}$-algebras, Property (UDE) is preserved by pushouts, i.e., if in a pushout diagram 
\begin{equation}\label{puUDE}
\xymatrix@C=20pt{
 A \ar[rr]^{\varphi_2}\ar[d]_{\varphi_1} && B_2 \ar[d]^{\varkappa_2}  \\
 B_1\ar[rr]_{\varkappa_1}&& B_1\ast_A B_2
 }
 \end{equation}
the homomorphism $\varphi_1$ satisfies Property (UDE), then  $\varkappa_2$ so does.
\end{pr}

For the proof we need an auxiliary proposition.
Note that every $( B_1\mathbin{\ast_A}B_2)$-$\mathop{\widehat\otimes}$-bimodule is both a $B_1$-$\mathop{\widehat\otimes}$- and $B_2$-$\mathop{\widehat\otimes}$-bimodule with the multiplications given by restriction of the canonical homomorphisms $\varkappa_1$ and $\varkappa_2$.

\begin{pr}\label{univprder}
Let $\varphi_1\!: A\to B_1$ and $\varphi_2\!: A\to B_2$ be
$\mathop{\widehat\otimes}$-algebra homomorphisms and~$X$ be a $(B_1\mathbin{\ast_A} B_2)$-$\mathop{\widehat\otimes}$-bimodule. Suppose  that
$D_1\in \Der(B_1,X)$, $D_2\in \Der(B_2,X)$ and $D_1\varphi_1=D_2\varphi_2$.  Then there exists a unique
continuous derivation $D$   making the diagram
\begin{equation*}
\xymatrix@C=20pt{
   A \ar[rr]^{\varphi_2}\ar[d]_{\varphi_1} && B_2 \ar[d]_{\varkappa_2}\ar[ddr]^{D_2}&  \\
 B_1\ar[rr]^{\varkappa_1}\ar[drrr]_{D_1}&& B_1\mathbin{\ast_A} B_2\ar@{-->}[dr]^{\!\!D }\\
     &&&X
 }
 \end{equation*}
 commutative.
\end{pr}
\begin{proof}
Consider the locally convex space
$(B_1\mathbin{\ast_A}B_2)\oplus X$ as a $\mathop{\widehat\otimes}$-algebra by letting
$$
(c,x)(c',x')=(cc',cx'+xc')\qquad (c,c'\in B_1\mathbin{\ast_A} B_2,\; x,x'\in X).
$$
For $j=1,2$, define a $\mathop{\widehat\otimes}$-algebra homomorphism by
$$
\psi_j\!:B_j\to (B_1\mathbin{\ast_A} B_2)\oplus X
\!:b\mapsto (\varkappa_j(b), D_j(b)).
$$
By the pushout property, there is  a unique homomorphism
$$
\psi\!: B_1\mathbin{\ast_A} B_2\to (B_1\mathbin{\ast_A} B_2)\oplus X
$$
such that  $\psi\varkappa_j=\psi_j$
($j=1,2$). Write $\psi$ as $c\mapsto (\alpha(c),D(c))$, where $\alpha$ and $D$ are continuous linear maps. It is easy to see that $\alpha$ is an endomorphism of $B_1\mathbin{\ast_A} B_2$ and $D$ is an
$\alpha$-derivation, i.e., $D(cc')=\alpha(c)D(c')+D(c)\alpha(c')$ ($c,c'\in B_1\mathbin{\ast_A} B_2$). Since $\alpha\varkappa_j=\varkappa_j$  ($j=1,2$),
the uniqueness of $\psi$ implies the uniqueness of $\al$. Hence $\alpha=1$. Therefore $D$ is a derivation such that $D \varkappa_j=D_j$ ($j=1,2$). Finally, $D$ is unique since so is~$\psi$.
\end{proof}

\begin{proof}[Proof of Proposition~\ref{DDPsta}]
Consider the pushout diagram in~\eqref{puUDE} and suppose that $\varphi_1$ satisfies Property (UDE). 
Take a $(B_1\ast_A B_2)$-$\mathop{\widehat\otimes}$-bimodule $X$  and $D_2\in
\Der(B_2,X)$.   Since $D_2\varphi_2$ is a derivation of~$A$ and $\varphi_1$
satisfies Property (UDE), there exists   $D_1\in
\Der(B_1,X)$ such that $D_1\varphi_1=D_2\varphi_2$. By  the universal
property for derivations in Proposition~\ref{univprder}, there is  $D\in
\Der(B_1\ast_A B_2,X)$  such that $D_1=D\varkappa_1$ and
$D_2=D\varkappa_2$. Hence,  $\Der(B_1\ast_A B_2,X)\to \Der(B_2,X)$ is
surjective.

Note that $\widetilde\varphi_X$ is injective for each $X$  if and only if $\varphi$ is an epimorphism; cf. the Fr\'echet algebra case in \cite[Theorem 3.20]{AP}. Since epimorphisms are preserved by pushouts in every category, $\Der(B_1\ast_A B_2,X)\to\Der(B_2,X)$ is injective for each~$X$.
Thus, $\varkappa_2$ satisfies Property (UDE).
\end{proof}

Now we apply Proposition~\ref{DDPsta} to analytic smash products.

\begin{thm}\label{UDEinh}
Let $\be\!:H\to K$ be a Hopf $\mathbin{\widehat{\otimes}}$-algebra homomorphism, $R$ and $S$ be an $H$- and $K$-$\mathbin{\widehat{\otimes}}$-module algebras, resp., and $\al\!: R\to S$ be an $H$-$\mathbin{\widehat{\otimes}}$-module algebra homomorphism. If each of $\al$ and $\be$ has dense range and satisfies Property (UDE), then so is the $\mathbin{\widehat{\otimes}}$-algebra homomorphism  $\al\mathop{\widehat{\#}}\be\!:R \mathop{\widehat{\#}} H \to S \mathop{\widehat{\#}} K$ defined in Lemma~\ref{homsmpr}.
\end{thm}
\begin{proof}
The density obviously inherits. Further, write $\al\mathop{\widehat{\#}} \be$ as the following composition:
$$
R \mathop{\widehat{\#}} H\xrightarrow{\al\mathop{\widehat{\#}} 1}  S \mathop{\widehat{\#}} H\xrightarrow{1\mathop{\widehat{\#}}\be}  S \mathop{\widehat{\#}} K.
$$
It is easy to see that Property (UDE) is stable under composition. So it suffices to show that  $\al\mathop{\widehat{\#}} 1$ and $1\mathop{\widehat{\#}}\be$ have this property. Since Property (UDE)  is preserved by pushouts according to Proposition~\ref{DDPsta}, the result follows from Propositions~\ref{UDElpu} and~\ref{UDErpu}, which assert that $\al\mathop{\widehat{\#}} 1$ and $1\mathop{\widehat{\#}}\be$ are obtained by pushouts with $\al$ and $\be$, respectively.
\end{proof}

\section{Homological epimorphisms  and relatively quasi-free algebras}
\label{s:rsfhe}

\subsection*{Definitions and statement of main result}

Our main reference on the homological theory of topological algebras is  Helemskii's book~\cite{X1} (the Russian edition is known as `first black book').  Additional facts regarding a more general relative theory can be found in~\cite{Pir_qfree}.

Suppose that $R$ is a $\mathop{\widehat\otimes}$-algebra. Recall that an \emph{$R$-$\mathop{\widehat\otimes}$-algebra} is a pair $(A,\eta_A)$, where $A$ is a $\mathop{\widehat\otimes}$-algebra and $\eta_A \!: R \to A$ is a $\mathop{\widehat\otimes}$-algebra homomorphism. Note that each $A$-$\mathop{\widehat\otimes}$-module is automatically an $R$-$\mathop{\widehat\otimes}$-module via the restriction-of-scalars functor along $\eta_A$. We denote by $(A,R){\mbox{-}\!\mathop{\mathsf{mod}}}$ the category of left $A$-$\mathop{\widehat\otimes}$-modules endowed with the split exact structure relative to~$R$. This means that the admissible sequences in $(A,R){\mbox{-}\!\mathop{\mathsf{mod}}}$  are those that are split by $R$-$\mathop{\widehat\otimes}$-module morphisms; cf. \cite[Appendix, Example 10.1 and 10.3]{Pir_qfree}. In particular, when $R = \mathbb{C}$, we recover the standard definition of an admissible (or $\mathbb{C}$-split) sequence of $A$-$\mathop{\widehat\otimes}$-modules used in \cite{X1}. When considering $\mathop{\widehat\otimes}$-bimodules over  an $R$-$\mathop{\widehat\otimes}$-algebra $A$ (from the left) and an $S$-$\mathop{\widehat\otimes}$-algebra $B$ (from the right), we denote the corresponding category by  $(A,R){\mbox{-}\!\mathop{\mathsf{mod}}\!\mbox{-}}(B,S)$. In the case when $R=S=\mathbb{C}$,  we write simply  $A{\mbox{-}\!\mathop{\mathsf{mod}}\!\mbox{-}} B$.
When we talk on a projective resolution in $(A,R){\mbox{-}\!\mathop{\mathsf{mod}}\!\mbox{-}}(B,S)$, we mean a complex consisting of objects projective in the corresponding exact category and splitting in $R{\mbox{-}\!\mathop{\mathsf{mod}}\!\mbox{-}} S$.

Let $A$ be a $\mathop{\widehat\otimes}$-algebra and let $X$ and $Y$ be a right and left $A$-$\mathop{\widehat\otimes}$-module, respectively. Then  the \emph{projective $A$-module tensor product} $X \ptens{A} Y$ is defined as the completion of the quotient space of $X \mathop{\widehat\otimes} Y$ by the closure of the linear hull of all elements of the form $x\cdot a\otimes y-x\otimes a\cdot y$ ($x\in X$, $y\in Y$, $a\in A$); see~\cite{X1}.

Recall that an \emph{$R$-$S$-homomorphism} from an $R$-$\mathop{\widehat\otimes}$-algebra to an $S$-$\mathop{\widehat\otimes}$-algebra is defined as a pair $(f,g)$, where $f\!: A \to B$  and $g\!: R \to S$ are $\mathop{\widehat\otimes}$-algebra homomorphisms, such that the diagram
 \begin{equation*}
\xymatrix@C=20pt{
A\ar[rr]^f && B  \\
R\ar[rr]^g\ar[u]^{\eta_A}&& S\ar[u]^{\eta_B}
 }
 \end{equation*}
is commutative. In the case when $S=R$, we obtain an \emph{$R$-homomorphism}.

\begin{df}\label{twsrelhoep}
\cite[Definition 6.2]{Pir_qfree}
An $R$-$S$-homomorphism $f\!: A \to B$ from an $R$-$\mathop{\widehat\otimes}$-algebra $A$ to an $S$-$\mathop{\widehat\otimes}$-algebra $B$ is called a \emph{two-sided relative homological epimorphism} if $f$ is an epimorphism
of $\mathop{\widehat\otimes}$-algebras and $A$ is acyclic relative to the functor
$$B\ptens{A}(-)\ptens{A}B\!: (A,R){\mbox{-}\!\mathop{\mathsf{mod}}\!\mbox{-}}(A,R)\to (B,S){\mbox{-}\!\mathop{\mathsf{mod}}\!\mbox{-}}(B,S),$$
i.e., it sends some (equivalently, each) projective resolution of $A$ in $(A,R){\mbox{-}\!\mathop{\mathsf{mod}}\!\mbox{-}}(A,R)$ to a sequence that splits in $S{\mbox{-}\!\mathop{\mathsf{mod}}\!\mbox{-}} S$.
\end{df}
We omit the definitions of \emph{left and right relative homological epimorphisms}; for details see \cite[Definition 6.1]{Pir_qfree} with corrections in \cite{Pir_co1,Pir_co2}. The only fact that we need here is that a two-sided relative homological epimorphism is also left and right \cite[Proposition 6.5]{Pir_qfree}.

In the case when $R=S=\mathbb{C}$,  Definition~\ref{twsrelhoep} takes more simple form:
\begin{df}\label{Homoepi}
An epimorphism $A\to B$ of $\mathop{\widehat\otimes}$-algebras is said to be \emph{homological} if $A$ is acyclic relative to the functor
\begin{equation}
\label{tenprfu}
B\ptens{A}(-)\ptens{A}B\!: A{\mbox{-}\!\mathop{\mathsf{mod}}\!\mbox{-}} A\to B{\mbox{-}\!\mathop{\mathsf{mod}}\!\mbox{-}} B;
\end{equation}
see \cite[Definition 1.3]{T2}, \cite[Definition 3.2]{Pir_stbflat} or \cite[Definition 6.3]{Pir_qfree}.
\end{df}

The following theorem is our main result.

\begin{thm}\label{AMheiffsol}
Let $\mathfrak{g}$ be a finite-dimensional complex Lie algebra.  Then the Arens-Michael envelope
$U(\mathfrak{g})\to \widehat U(\mathfrak{g})$ is a homological epimorphism if and only if~$\mathfrak{g}$ is solvable.
\end{thm}

For the proof we need a result on homological epimorphisms for analytic smash products with the Hopf $\mathop{\widehat\otimes}$-algebra $\mathfrak{A}_s$  defined in~\eqref{faAsdef} but in a more general formulation; see Theorem~\ref{Assmpr} below.

\subsection*{Smash products with $\mathfrak{A}_s$}

The following proposition is the first step of the proof of Theorem~\ref{AMheiffsol}.

\begin{pr}\label{CzAshome}
Let $s\in[0,+\infty]$. Then $\be\!:\mathbb{C}[z]\to \mathfrak{A}_s$ in Lemma~\ref{CzAsHh} is a homological epimorphism.
\end{pr}
\begin{proof}
Since $\mathfrak{A}_s$ is a Hopf $\mathop{\widehat\otimes}$-algebra by Lemma~\ref{CzAsHh}, it follows from
\cite[Proposition~3.7]{Pir_stbflat} that it suffices to check that $\mathbb{C}[z]\to \mathfrak{A}_s$ is a weak localization, i.e.,
the following conditions hold:

(1)~the natural map $\mathfrak{A}_s\ptens{\mathbb{C}[z]}\mathbb{C}\to \mathbb{C}$ is a topological isomorphism;

(2)~for some  projective resolution $0 \leftarrow \mathbb{C} \leftarrow P_\bullet$ in $\mathbb{C}[z]{\mbox{-}\!\mathop{\mathsf{mod}}}$ the complex $$0\leftarrow \mathfrak{A}_s\ptens{\mathbb{C}[z]}\mathbb{C} \leftarrow \mathfrak{A}_s\ptens{\mathbb{C}[z]} P_ \bullet $$ is admissible.

To check Condition~(1) recall that  for a $\mathop{\widehat\otimes}$-algebra $A$, a right Fr\'echet module $X$, and a closed ideal $I$ such that $A/I$ is also a Fr\'echet space the well-known formula $X\ptens{A}(A/I) \cong X/\,\overline{X\cdot I}$ holds; see, e.g.,  \cite[Proposition 3.1]{Pir_flcyc}. When $A=\mathbb{C}[z]$, $X=\mathfrak{A}_s$ and $I$ is the ideal of polynomials vanishing at~$0$, we have the desired isomorphism $\mathfrak{A}_s\ptens{\mathbb{C}[z]}\mathbb{C}\to \mathbb{C}$.

To check Condition~(2) note that
\begin{equation*}
0 \leftarrow \mathbb{C} \xleftarrow{\varepsilon} \mathbb{C}[z]\xleftarrow{d_0}\mathbb{C}[z]\leftarrow 0,
\end{equation*}
where $\varepsilon\!:z\mapsto 0$ and $[d_0(a)](z)\!:=az$, is a projective resolution in $\mathbb{C}[z]{\mbox{-}\!\mathop{\mathsf{mod}}}$. Applying $\mathfrak{A}_s\ptens{\mathbb{C}[z]}(-)$ to this resolution we get
\begin{equation}\label{Asres}
0 \leftarrow \mathbb{C} \longleftarrow \mathfrak{A}_s\longleftarrow\mathfrak{A}_s\leftarrow 0,
\end{equation}
where the differentials are defined by the same formulas.

Consider the map
$$
\si\!:\mathfrak{A}_s\to \mathfrak{A}_s\!:a\mapsto \sum_{n=0}^\infty a_{n+1}z^n.
$$
Suppose that $s<\infty$ and let $C$ be a positive constant such that $n^s\leqslant C\,2^n$ for every $n\in\mathbb{Z}_+$. Then for each $r>0$ we have
$$
\|\si(a)\|_{r,s}=\sum_{n=0}^\infty |a_{n+1}|\,\frac{r^n}{n!^s}\leqslant \sum_{n=0}^\infty \frac{n^s}{r}  \, |a_{n}|\,\frac{r^n}{n!^s}\leqslant  \frac{C}{r}\,  \|a\|_{2r,s}.
$$
Therefore $\si$ is well defined and continuous. It is easy to see that the sequence in~\eqref{Asres} splits by $\si$ and hence it is admissible. The case when $s=\infty$ is obvious.

Thus both condition are satisfied and therefore $\mathbb{C}[z]\to \mathfrak{A}_s$ is a weak localization for every~$s$.
\end{proof}

Recall that when an algebra $R$ is endowed with a derivation, the action induced by this derivation turns~$R$ into a $\mathbb{C}[z]$-module algebra and so one can consider the smash product $R{\mathop{\#}} \mathbb{C}[z]$.

\begin{thm}\label{Assmpr}
Let $K$ be a Hopf $\mathbin{\widehat{\otimes}}$-algebra, $R$ a $\mathbb{C}[z]$-module algebra of countable dimension and $S$ a $K$-$\mathbin{\widehat{\otimes}}$-module algebra. Suppose that $\be\!:\mathbb{C}[z]\to K$ is a Hopf $\mathbin{\widehat{\otimes}}$-algebra homomorphism and
$\al\!:R\to S$ is a $\mathbin{\widehat{\otimes}}$-algebra homomorphism that is a $\mathbb{C}[z]$-module morphism. If $\al$ and $\be$ are homological epimorphisms with dense range, then so is
$$
\al\mathop{\widehat{\#}}\be\!:R{\mathop{\#}} \mathbb{C}[z]\to S{\mathop{\widehat{\#}}} K.
$$
\end{thm}

To prove Theorem~\ref{Assmpr} we establish a generalization of Theorem~\ref{AMerelhoep} and next apply  Theorem~\ref{PiTh}.

Recall that an $R$-$\mathop{\widehat\otimes}$-algebra $A$ is said to be \emph{relatively quasi-free} over $R$ if any admissible singular $R$-extension of $A$ is split; see \cite[Definition 7.1]{Pir_qfree} and the discussion therein.

\begin{thm}\label{Brelhoep}
Let $(f,g)$ be an $R$-$S$-homomorphism from an $R$-$\mathop{\widehat\otimes}$-algebra $A$ to a $S$-$\mathop{\widehat\otimes}$-algebra $B$. Suppose that

\emph{(1)}~$g\!: R \to S$ has dense range;

\emph{(2)}~$f\!: A \to B$ satisfies Property (UDE) and  has dense range;

\emph{(3)}~$A$ is relatively quasi-free over $R$.

Then~$f$ is a two-sided relative homological epimorphism.
\end{thm}

For the proof we need a relative version of Property (UDE). If $A$ is an $R$-$\mathop{\widehat\otimes}$-algebra and $X$ is an $A$-$\mathop{\widehat\otimes}$-bimodule, then a derivation $D\!: A \to X$ is called an \emph{$R$-derivation} if it is an  $R$-bimodule morphism  (or, equivalently, a left or right $R$-module morphism). The subspace of $\Der(A, X)$ consisting of
$R$-derivations is denoted by $\Der_R(A, X)$.
An $R$-$S$-homomorphism $(f,g)\!: (A,R)\to (B,S)$ induces the linear map
$$
\Der_S(B,X)\to \Der_R(A,X)\!:D\mapsto Df
$$
for every $B$-$\mathop{\widehat\otimes}$-bimodule $X$;
see details in~\cite{Pir_qfree}.

\begin{pr}\label{GUDEA}
Let $(f,g)\!: (A,R)\to (B,S)$ be an $R$-$S$-homomorphism from an $R$-$\mathop{\widehat\otimes}$-algebra $A$ to an $S$-$\mathop{\widehat\otimes}$-algebra $B$. Suppose that each of $f$ and $g$ has dense range and $f$ satisfies Property (UDE). Then the canonical map $\Der_S(B,X)\to \Der_R(A,X)$ is a bijection
for every $B$-$\mathop{\widehat\otimes}$-bimodule~$X$.
\end{pr}
\begin{proof}
Since $\Der(B,X)\to \Der(A,X)$ is  injective, so is $\Der_S(B,X)\to \Der_R(A,X)$.

To prove the surjectivity take $D\in \Der_R(A,X)$. Since $f$ satisfies Property (UDE), there is $D'\in \Der(B,X)$ such that $D=D'f$. To complete the proof we need to show that $D'$ is an $S$-derivation, that is,
\begin{equation}\label{Sder}
D'(\eta_B(s)b)=\eta_B(s)\cdot D'(b)\qquad\text{for every $s\in S$ and $b\in B$}.
\end{equation}

Since each of $f$ and $g$ has dense range, it suffices to verify~\eqref{Sder}  in the case when $s=g(r)$ and $b=f(a)$ for some $r\in R$ and $a\in A$. Note that
$$\eta_B(s)b=\eta_B(g(r))f(a)=f(\eta_A(r))f(a)=f(\eta_A(r)a)$$ and hence $D'(\eta_B(s)b)=D(\eta_A(r)a)$. Since $D$ is an $R$-derivation, we have $D'(\eta_B(s)b)=r\cdot D(a)$. On the other hand,
$$
r\cdot D(a)=f(\eta_A(r))\cdot D(a)=\eta_B(g(r))\cdot D'f(a)=\eta_B(s)\cdot D'(b).
$$
Thus \eqref{Sder} holds.
\end{proof}
\begin{rem}
The conditions of Proposition~\ref{GUDEA} can be weakened by assuming that $g$ is an epimorphism of $\mathop{\widehat\otimes}$-algebras;
cf.~\cite[Proposition~3.8]{Pir_qfree}. But we do not need this variant.
\end{rem}

If $A$ is an $R$-$\mathop{\widehat\otimes}$-algebra, then there is a representing object of the functor $\Der_R(A,-)$ from $A{\mbox{-}\!\mathop{\mathsf{mod}}\!\mbox{-}} A$ to the category of sets. It is denoted by $\Omega^1_R A$ and called
the \emph{bimodule of relative differential $1$-forms} of $A$; for details see \cite[p.\,82]{Pir_qfree}.

Fix an $R$-$S$-homomorphism and a $B$-$\mathop{\widehat\otimes}$-bimodule~$X$.
From the universal property of $\Omega^1$ we have the commutative diagram
\begin{equation}\label{derprdi}
\xymatrix@C=20pt{
 \Der_S(B,X) \ar[rr]\ar@{=}[d] && \Der_R(A,X) \ar@{=}[d]  \\
 {_B}{\mathbf h}{_B}(\Omega_S^1 B,X)\ar[rr]
 &&{_A}{\mathbf h}{_A}(\Omega_R^1 A,X),
 }
\end{equation}
where ${\mathbf h}{_B}(-,-)$ denotes the vector space of $B$-$\mathop{\widehat\otimes}$-bimodule morphisms (and similarly for $A$); see \cite[eq. (7.3) and p.\,83]{Pir_qfree}. Also, consider the natural map
\begin{equation}\label{hORaBBAB}
{_A}{\mathbf h}{_A}( \Omega_R^1 A,X)\to {_B}{\mathbf h}{_B}(B\ptens{A} \Omega^1_R A\ptens{A}B,X)
\end{equation}
and the composition with the bottom arrow,
\begin{equation}\label{bhbm}
{_B}{\mathbf h}{_B}( \Omega_S^1 B,X)\to {_B}{\mathbf h}{_B}(B\ptens{A} \Omega^1_R A\ptens{A}B,X).
\end{equation}
Substituting $\Omega_S^1 B$ for $X$, we obtain a $B$-$\mathop{\widehat\otimes}$-bimodule morphism
\begin{equation}\label{oAoB}
B\ptens{A} \Omega^1_R A\ptens{A}B\to \Omega_S^1 B
\end{equation}
corresponding to the identity morphism
of $\Omega_S^1 B$.

\begin{pr}\label{GUDEB}
\emph{(cf.~\cite[Proposition~7.4]{Pir_qfree})}
Let $(f,g)\!: (A,R)\to (B,S)$ be an $R$-$S$-homomorphism from an $R$-$\mathop{\widehat\otimes}$-algebra $A$ to an $S$-$\mathop{\widehat\otimes}$-algebra $B$. Then the canonical map $\Der_S(B,X)\to \Der_R(A,X)$ is a bijection
for every $B$-$\mathop{\widehat\otimes}$-bimodule~$X$ if and only if the morphism  in~\eqref{oAoB} is a $B$-$\mathop{\widehat\otimes}$-bimodule isomorphism.
\end{pr}

\begin{proof}
Note that~\eqref{oAoB} is an isomorphism if and only if~\eqref{bhbm} is bijective
for every $B$-$\mathop{\widehat\otimes}$-bimodule~$X$. Since~\eqref{hORaBBAB} is always bijective, \eqref{bhbm} is bijective if and only if the bottom arrow in~\eqref{derprdi} is bijective.
This  is clearly equivalent to the bijectivity of the top arrow in\eqref{derprdi}.
\end{proof}

\begin{proof}[Proof of Theorem~\ref{Brelhoep}]
By \cite[Proposition 7.3]{Pir_qfree}, $A$ is relatively quasi-free over $R$ if and only if $\Omega^1_RA$ is projective in  $(A, R){\mbox{-}\!\mathop{\mathsf{mod}}\!\mbox{-}}(A, R)$. Hence the sequence
 \begin{equation}
\label{Omega}
0 \leftarrow A  \xleftarrow{\mu_A} A\ptens{R} A \xleftarrow{j} \Omega^1_R A \leftarrow 0
\end{equation}
is a projective resolution in $(A, R){\mbox{-}\!\mathop{\mathsf{mod}}\!\mbox{-}}(A, R)$. Here $\mu_A$ is the multiplication on $A$ and $j$ is the bimodule morphism corresponding to the inner derivation 
$$
-{\mathop{\mathrm{ad}}\nolimits}_{1\otimes 1}\!: A \to A\ptens{R}A\!: a \mapsto  1\otimes a-a\otimes 1.
$$

Since  $f$ has dense range, it is an epimorphism. Then the canonical map $B\ptens{A} B\to B$ is a topological isomorphism. Also, since $g$  is an epimorphism, the canonical map $B\ptens{R} B\to B\ptens{S} B $  is a topological isomorphism; for both properties of  epimorphisms see \cite[Proposition 6.1]{Pir_qfree}. 

We claim that the application of the functor $B\ptens{A}(-)\ptens{A}B$ to the sequence in~\eqref{Omega}  produces the complex
\begin{equation}\label{OmBsh}
0 \leftarrow B  \leftarrow  B\ptens{S} B \leftarrow \Omega^1_S B \leftarrow 0.
\end{equation}
Indeed, it suffices to show that $B\ptens{A} \Omega^1_R A\ptens{A}B\cong \Omega_S^1 B$.
It follows from Proposition~\ref{GUDEA} that the top arrow in~\eqref{derprdi}  is bijective for every $X$ and so by Proposition~\ref{GUDEB}, we have an isomorphism.

Finally, \eqref{OmBsh} splits in $B{\mbox{-}\!\mathop{\mathsf{mod}}\!\mbox{-}} S$ by \cite[Proposition 7.2]{Pir_qfree}. Hence it also splits in $S{\mbox{-}\!\mathop{\mathsf{mod}}\!\mbox{-}} S$. This implies that $f$ is a two-sided relative homological epimorphism; see Definition~\ref{twsrelhoep}.
\end{proof}

Following \cite[Definition 8.3]{Pir_qfree}, we say that a projective $(A, R)$-$\mathop{\widehat\otimes}$-bimodule $P$ satisfies the \emph{finiteness condition $(f_2)$} if it is a retract of a bimodule of the form $A\ptens{R}R^n\ptens{R} A$ for some $n\in\mathbb{N}$. (We do not need the alternative condition $(f_1)$ from \cite{Pir_qfree}.) Also, an $R$-$\mathop{\widehat\otimes}$-algebra $A$ is said to be of \emph{$(f_2)$-finite type} if in $(A, R){\mbox{-}\!\mathop{\mathsf{mod}}\!\mbox{-}}(A, R)$ it has a  projective finite-length resolution all of whose bimodules satisfy the finiteness condition $(f_2)$ \cite[Definition 8.4]{Pir_qfree}.

\begin{pr}
\label{Oeqfft}
Let $R$ be an associative algebra of countable dimension and~$\de$ a derivation of~$R$. Denote the Ore extension $R[z; \de]$ by~$A$. Then the following conditions hold.

\emph{(A)}~$A$ is relatively quasi-free over $R$.

\emph{(B)}~$A$ is an $R$-$\mathop{\widehat\otimes}$-algebra of $(f_2)$-finite type.
\end{pr}
\begin{proof}
Part~(A) is a partial case of Proposition~7.9 in \cite{Pir_qfree}. Moreover, since we deal with an Ore extension with trivial twisting (i.e., the twisting endomorphism is identical), this proposition implies
that $\Omega^1_R A \cong A\mathbin{\otimes_R} A$ as an $(A, R)$-bimodule; see also \cite[Proposition~7.8]{Pir_qfree}. In particular, $\Omega^1_R A$ satisfies the finiteness condition $(f_2)$. Since $A$ is relatively quasi-free by Part~(A), it follows from \cite[Proposition~7.3.]{Pir_qfree} that $\Omega^1_R A$ is a projective $(A, R)$-bimodule. Then the sequence in~\eqref{Omega} is a projective resolution and thus $A$ is of $(f_2)$-finite type; cf. \cite[Example~8.1]{Pir_qfree}. This completes the proof of Part~(B).
\end{proof}

The following result gives a direct relation between homological epimorphisms and Property (UDE).

\begin{thm}\label{homeUDE}
Every homological epimorphism of $\mathop{\widehat\otimes}$-algebras satisfies Property (UDE).
\end{thm}

\begin{proof}
Let $\varphi\!:A\to B$ be a homological epimorphism. Consider the (unnormalized) bimodule bar resolution of $A$ \cite[Section III.2.3]{X1}. It has the form
\begin{equation}
\label{barA}
0 \leftarrow A \xleftarrow{\mu_A} A\mathop{\widehat\otimes} A \xleftarrow{d_1} A\mathop{\widehat\otimes} A\mathop{\widehat\otimes} A \xleftarrow{d_2} \cdots
\leftarrow A^{\mathop{\widehat\otimes} n} \leftarrow \cdots,
\end{equation}
where $\mu_A$ is the multiplication map and $d_n$, $n\in\mathbb{N}$ are 
is given by standard formulas. 
Applying $B\ptens{A}(-)\ptens{A}B$ to this resolution  we get the sequence
\begin{equation*}
0 \leftarrow B \xleftarrow{\mu_{B}} B\mathop{\widehat\otimes} B \xleftarrow{\widetilde d_1} B\mathop{\widehat\otimes} A\mathop{\widehat\otimes} B \xleftarrow{\widetilde d_2} \cdots,
\end{equation*}
which is admissible since $\varphi$ is a homological epimorphism. Hence $\Coker \widetilde d_2\to \Ker\mu_B$ is a topological isomorphism.

Note that $\Ker\mu_A\cong \Omega^1A$ and $\Ker\mu_B\cong \Omega^1B$
\cite[Proposition 7.2]{Pir_qfree} and, moreover, $\Coker  d_2\cong \Omega^1A$ because~\eqref{barA} is admissible. Since the functor of relative tensor product preserves cokernels \cite[Proposition 3.3]{Pir_flcyc}, we have that  $B\ptens{A} \Omega^1 A  \ptens{A} B \to \Omega^1 B$ defined in~\eqref{oAoB} is a $B$-$\mathop{\widehat\otimes}$-bimodule isomorphism.  Applying Proposition~\ref{GUDEB} with $R=S=\mathbb{C}$, we conclude that $\varphi$  satisfies Property (UDE).
\end{proof}

Now we can prove the main result of the section.

\begin{proof}[Proof of Theorem~\ref{Assmpr}]
Recall that we denote by $\be$ the Hopf $\mathop{\widehat\otimes}$-algebra homomorphism $\mathbb{C}[z]\to K$.
Put $A=R{\mathop{\#}} \mathbb{C}[z]$, $B=S{\mathop{\widehat{\#}}} K$, $g=\al$ and $f=\al\mathop{\widehat{\#}}\be$ (the latter is well defined by Lemma~\ref{homsmpr} since $\al$ is a $\mathbb{C}[z]$-module morphism). It suffices to check the  conditions in Theorem~\ref{PiTh}, i.e.,

(1)~$g$ is a homological epimorphism;

(2)~$f$ is a left or right relative homological epimorphism;

(3)~$A$ is projective in $R{\mbox{-}\!\mathop{\mathsf{mod}}}$ and $B$ is projective in $S{\mbox{-}\!\mathop{\mathsf{mod}}}$;

(4)~$A$ is an $R$-$\mathop{\widehat\otimes}$-algebra of $(f_2)$-finite type.

Condition~(1) holds by the hypothesis.

Condition~(2) follows from Theorem~\ref{Brelhoep}.
Indeed, since $\al$ and $\be$ are homological epimorphisms, they satisfy Property (UDE)   by  Theorem~\ref{homeUDE}. Furthermore, Theorem~\ref{UDEinh} implies that $f$ also satisfies Property (UDE). Since $R$ is a $\mathbb{C}[z]$-module algebra, we have that the action is induced by a derivation. Therefore $R{\mathop{\#}} \mathbb{C}[z]$ is an Ore extension (see Remark~\ref{OeAsp}) and hence it is relatively quasi-free by Part~(A) of Proposition~\ref{Oeqfft}.  Thus all the conditions in Theorem~\ref{Brelhoep} hold and therefore $f$ is a two-sided relative homological epimorphism. Finally, note that every two-sided relative homological epimorphism is a one-sided relative homological epimorphism \cite[Proposition 6.5]{Pir_qfree}.

To verify Condition~(3) note that  $S{\mathop{\widehat{\#}}} K$ is isomorphic to $S\mathop{\widehat\otimes} K$ as an $S$-$\mathop{\widehat\otimes}$-module and so it is projective being a free module (see, e.g., \cite[Chapter III, Proposition 1.25]{X1}). Similarly, $R{\mathop{\#}} \mathbb{C}[z]$ is a projective $R$-module.

Finally, Condition~(4) holds by Part~(B) of Proposition~\ref{Oeqfft}.
\end{proof}

\subsection*{The proof of the main theorem}
Now we are in a position to prove the central result of the paper.

\begin{proof}[Proof of Theorem~\ref{AMheiffsol}]
For the necessity see \cite[Theorem~3.6]{Pi4}.

To show the sufficiency  consider the decomposition in~\eqref{expfsmp} from Theorem~\ref{AMede}. We claim that the truncated homomorphism
\begin{equation}\label{truAMe}
(\cdots (U(\mathfrak{f}_1) \mathop{\#}U(\mathfrak{f}_2))\mathop{\#} \cdots)\mathop{\#} U(\mathfrak{f}_ k)\longrightarrow(\cdots (\mathfrak{A}_{i_1} \mathop{\widehat{\#}}\mathfrak{A}_{i_2})
\mathop{\widehat{\#}}\cdots)\mathop{\widehat{\#}}\mathfrak{A}_{i_k}
\end{equation}
(which exists since the decompositions of $U(\mathfrak{g})$ and $\widehat U(\mathfrak{g})$ are compatible)
is a homological epimorphism for every $k=1,\ldots,n$.

Proceeding by induction, note first that $U(\mathfrak{f}_k)\cong \mathbb{C}[z]$ for every $k$.
If $k=1$, then the assertion is given by Proposition~\ref{CzAshome}.

Assume now that the claim holds for $k-1$.
Then the homomorphism in \eqref{truAMe} can be written as $$\al\mathop{\widehat{\#}}\be\!:R{\mathop{\#}} \mathbb{C}[z]\to S{\mathop{\widehat{\#}}} \mathfrak{A}_{i_k},$$ where $R$ and $S$ denote the algebras obtained on the $(k-1)$th step. Here  $S$ is a $\mathfrak{A}_{i_k}$-$\mathbin{\widehat{\otimes}}$-module algebra, $\be\!:\mathbb{C}[z]\to \mathfrak{A}_{i_k}$ is a Hopf $\mathbin{\widehat{\otimes}}$-algebra homomorphism and
$\al\!:R\to S$ is a $\mathbin{\widehat{\otimes}}$-algebra homomorphism that is a $\mathbb{C}[z]$-module morphism.

Note that $R$ has countable dimension and each of $\al$ and $\be$ has dense range. Moreover, $\al$ and $\be$ are homological epimorphisms (the  first by the induction hypothesis and  the second  by Proposition~\ref{CzAshome}). Thus all the assumptions in Theorem~\ref{Assmpr} hold and hence $R{\mathop{\#}} \mathbb{C}[z]\to S{\mathop{\widehat{\#}}} \mathfrak{A}_{i_k}$ is also a homological epimorphism. The claim is proved.

Finally, putting $k=n$ we conclude that $U(\mathfrak{g})\to \widehat U(\mathfrak{g})$ is a homological epimorphism.
\end{proof}

\subsection*{Acknowledgment}
A part of this work was done during the author's visit to the HSE University (Moscow) in the winter of 2024.  I wish to thank this university, and in particular Aleksei Pirkovskii, for the hospitality. I would like to expresses my gratitude to the referees for their suggestions that improve the presentation, as well as for pointing out inaccuracies in some arguments, especially in the proof of Proposition~\ref{pushoex}. Vladimir Al.~Osipov's valuable comments are also appreciated.


\begin{thebibliography}{99}

\bibitem{Ak08}
S.\,S.~Akbarov, \emph{Holomorphic functions of exponential type and duality for Stein groups with algebraic
connected component of identity},   J. Math. Sci., 162:4 (2009), 459--586.

\bibitem{ArAMN}
O.\,Yu.~Aristov, \emph{Arens-Michael envelopes of nilpotent Lie
algebras, functions of exponential type, and homological
epimorphisms}, Trans. Moscow Math. Soc. (2020), 97--114.

\bibitem{ArAnF}
O.\,Yu.~Aristov, \emph{Holomorphic functions of exponential type on connected complex Lie groups}, J. Lie Theory 29:4 (2019), 1045--1070.

\bibitem{ArRC}
O.\,Yu.~Aristov, \emph{The relation ``Commutator equals function'' in Banach algebras},  Math. Notes, 109:3 (2021), 323--334.

\bibitem{ArHR}
O.\,Yu.~Aristov, \emph{On holomorphic reflexivity conditions for complex Lie groups}, Proc. Edinb. Math. Soc.,  64:4  (2021). 800--821.

\bibitem{ArPiLie}
O.\,Yu.\,Aristov, \emph{When a completion of the universal enveloping algebra is a Banach PI-algebra?}, Bull. Aust. Math. Soc, 107:3 (2023), 493--501.

\bibitem{ArLfd}
O.\,Yu.\,Aristov,
\emph{Length functions exponentially distorted on subgroups of complex Lie groups},  European Journal of Mathematics, 9, (2023), 60.

\bibitem{AHHFG}
O.\,Yu.~Aristov, \emph{Holomorphically finitely generated Hopf algebras and quantum Lie groups}, Moscow Math. J. 24:2 (2024) 145--180.

\bibitem{Ar_smash}
O.\,Yu.\,Aristov, \emph{Decomposition of the algebra of analytic functionals on a connected complex Lie group and its completions into iterated analytic smash products},  Algebra i Analiz 36:4 (2024) in print  (Russian), arXiv: 2209.04192; English transl. to appear in St. Petersburg Math.~J.

\bibitem{AP}
O.\,Yu.~Aristov, A.\,Yu.~Pirkovskii, \emph{Open embeddings and pseudoflat epimorphisms}, J. Math. Anal. Appl. 485 (2020) 123817.

\bibitem{BBK18}
F.\,Bambozzi, O.\,Ben-Bassat, K.\,Kremnizer, \emph{Stein domains in Banach algebraic geometry}, J.
Funct. Anal. 274 (2018), no. 7, 1865--1927.

\bibitem{BerDic}
 G.\,M.~Bergman, W.\,Dicks \emph{Universal derivations and universal ring
constructions}, Pacific J. Math. 79 (1978), 293--337.

\bibitem{Co85}
A.\,Connes, \emph{Non-commutative differential geometry}, Inst. Hautes \'{E}tudes Sci. Publ. Math.
(1985), 41--144.

\bibitem{Do03}
A.\,A.~Dosiev (Dosi), \emph{Homological dimensions of the algebra formed by entire functions of elements of a nilpotent Lie algebra}, Funct. Anal. Appl. 37:1, (2003) 61--64.

\bibitem{Do09}
A.\,A.~Dosiev (Dosi), \emph{Local left invertibility for operator tuples and
noncommutative localizations}. J. K-Theory 4:1 (2009),  163--191.

\bibitem{GL}
W.\,Geigle, H.\,Lenzing,
\emph{Perpendicular categories with applications to representations and sheaves}.
J.~Algebra 144:2 (1991), 273--343.

\bibitem{Gl12}
H.\,Gl\"{o}ckner, \emph{Continuity of bilinear maps on direct sums
of topological vector spaces}, J. Funct. Anal. 262 (2012), 2013--2030.

\bibitem{X1}
A.\,Ya.~Helemskii,
\emph{The homology of Banach and topological algebras}.
Mathematics and its Applications (Soviet Series), 41.
Kluwer Academic Publishers Group, Dordrecht, 1989.

\bibitem{Kel}
J.\,L.~Kelley  \emph{General topology}, Van Nostrand, New York, Toronto, London, 1955.

\bibitem{Ku06}
Yu.\,V.~Kuzmin, \emph{Homological group theory}, Adv. Stud. Math. Mech. V.~1, Factorial Press,  Moscow, 2006, in Russian.

\bibitem{Lit}
G.\,L.~Litvinov, \emph{Group representations in locally convex spaces, and topological group
algebras}, Selecta Math. Soviet. (Birkh\"{a}user Publ.) 7:2 (1988), 101--182.

\bibitem{MR87}
J.\,C.~McConnell, J.\,C.~Robson,  \emph{Noncommutative Noetherian rings}, John Wiley, Chichester, 1987.

\bibitem{Me04}
R.\,Meyer, \emph{Embeddings of derived categories of bornological modules},
arXiv:041059  (2004).

\bibitem{Pi99}
A.\,Yu.~Pirkovskii, \emph{On certain homological properties of Stein algebras}, Functional analysis, 3. J.~Math. Sci. (NewYork) 95:6 (1999),  2690--2702.

\bibitem{Pi4}
A.\,Yu.~Pirkovskii, \emph{Arens--Michael enveloping algebras and analytic smash
products}, Proc. Amer. Math. Soc. 134 (2006), 2621--2631.

\bibitem{Pir_stbflat}
A.\,Yu.~Pirkovskii, \emph{Stably flat completions of universal enveloping algebras},
Dissertationes Math. (Rozprawy Math.) 441 (2006), 1--60.

\bibitem{Pir_qfree}
A.\,Yu.~Pirkovskii, \emph{Arens-Michael envelopes, homological epimorphisms, and relatively
quasi-free algebras}, Trans. Moscow Math. Soc. 2008, 27--104.

\bibitem{Pir_flcyc}
A.\,Yu.~Pirkovskii, \emph{Flat cyclic Fr\'{e}chet modules, amenable Fr\'{e}chet
algebras, and approximate identities}. Homology, Homotopy and Applications
11:1 (2009), 81--114.

\bibitem{Pi15}
A.\,Yu.~Pirkovskii, \emph{Holomorphically finitely generated
algebras}, J. Noncommutative Geom. 9:1 (2015), 215--264.

\bibitem{Pir_co1}
A.\,Yu.~Pirkovskii, \emph{A note on relative homological epimorphisms of topological algebras},
arXiv:2104.13716.

\bibitem{Pir_co2}
A.\,Yu.~Pirkovskii, \emph{Letter to the editors}, Trans. Moscow Math. Soc., 82 (2021), 327--328.

\bibitem{Sw69}
M.\,E.\,Sweedler, \emph{Hopf algebras},  Benjamin, N.Y., 1969.


\bibitem{T2}
J.\,L.~Taylor, \emph{A general framework for a multi-operator functional calculus},
Adv. Math. 9 (1972), 183--252.

\bibitem{Va01}
C.\,Valqui, \emph{Universal extension and excision for
topological algebras}, K-Theory, 22 (2001), 145--160.

\end{thebibliography}
\end{document}